\documentclass[11pt,letterpaper]{article}
\usepackage[margin=1in]{geometry}
\usepackage{amsthm}
\usepackage{amsmath,amssymb}
\usepackage{bbm}
\usepackage{mathtools}
\usepackage{amssymb}
\usepackage{hyperref}
\usepackage{cleveref}
\usepackage{tikz-cd}
\usepackage{tikz}
\usepackage{graphicx}
\usepackage{array}
\usetikzlibrary{decorations.markings}
\usepackage{enumitem}
\usepackage{amsfonts}
\usepackage{todonotes}
\usepackage{ytableau}
\usepackage{subcaption}
\usepackage[export]{adjustbox}
\usepackage{mathdots}
\usepackage{comment}
\usepackage{float}
\usepackage{pgfplots}
\pgfplotsset{compat=1.15}
\usepackage{mathrsfs}
\usetikzlibrary{arrows}

\newtheorem{theorem}{Theorem}[section]
\newtheorem{prop}[theorem]{Proposition}
\newtheorem{conj}[theorem]{Conjecture}
\newtheorem{lemma}[theorem]{Lemma}
\newtheorem{cor}[theorem]{Corollary}

\theoremstyle{definition}

\newtheorem{eg}[theorem]{Example}
\newtheorem{quest}[theorem]{Question}

\newtheorem{defin}[theorem]{Definition}
\newtheorem{remark}[theorem]{Remark}

\newtheorem{notation}[theorem]{Notation}

\newcommand{\rk}{\operatorname{rk}}

\newcommand{\col}{\operatorname{col}}

\newcommand{\cw}{\operatorname{cw}}
\newcommand{\ccw}{\operatorname{ccw}}
\newcommand{\pis}{{\pi,\sigma}}

\title{Characterizing positroid quotients of uniform matroids}

\author{Zhixing Chen\thanks{Fudan University, Shanghai \href{mailto:20307110030@fudan.edu.cn}{\tt 20307110030@fudan.edu.cn}}
\and  Yumou Fei\thanks{Peking University, Beijing \href{mailto:feiym2002@stu.pku.edu.cn}{\tt feiym2002@stu.pku.edu.cn}} \and Jiyang Gao\thanks{Harvard University, Cambridge, MA 02138 \href{mailto:jgao@math.harvard.edu}{\tt jgao@math.harvard.edu}}
\and Yuxuan Sun\thanks{University of Minnesota Twin Cities, Minneapolis, MN 55455 \href{mailto:sun00816@umn.edu}{\tt sun00816@umn.edu}}
\and Yuchong Zhang\thanks{University of Michigan, Ann Arbor, MI 48104 \href{mailto:zongxun@umich.edu}{\tt zongxun@umich.edu}}}

\begin{document}

\maketitle

\begin{abstract}

We study two-step flag positroids $(P_1, P_2)$, where $P_1$ is a quotient of $P_{2}$. We provide a complete characterization of all two-step flag positroids that contain a uniform matroid, extending and completing a partial result by Benedetti, Ch\'avez, and Jim\'enez. To contrast general positroids with the special case of lattice path matroids, we show that the containment relations of Grassmann necklaces and conecklaces fully characterize flag lattice path matroids, but are insufficient for general flag positroids. Additionally, we prove that the decorated permutations of any elementary quotient pair are related by a cyclic shift, resolving a conjecture of Benedetti, Ch\'avez and Jim\'enez.

\textbf{keywords:} positroids, flag positroids, matroid quotients, lattice path matroids, Grassmann necklaces, decorated permutations
\end{abstract}

\section{Introduction}
Positroids are an important class of matroids that can be represented by full-rank matrices with nonnegative maximal minors. They were first introduced by Postnikov in his study of the totally nonnegative Grassmannian \cite{Pos06}. The applications of positroids span several domains, including cluster algebra \cite{PS14} and physics \cite{KW11, Arkani21}.

Positroids also have many nice properties. They are closed under matroid duality and cyclic shifts of the ground set. Moreover, they are in bijection with many combinatorial objects, including Grassmann necklaces and decorated permutations \cite{Pos06,Oh11}.

An ordered pair of matroids $(M_1,M_2)$ is called a \emph{(two-step) flag matroid} if $M_1$ is a quotient of $M_2$, which means that every circuit of $M_2$ is a union of circuits of $M_1$. If both $M_1$ and $M_2$ are positroids, then the pair is called a \emph{(two-step) flag positroid}. 

Characterizing flag positroids via circuits is computationally challenging, as enumerating all unions of circuits requires exponential time. Since positroids can be represented concisely by many combinatorial objects such as Grassmann necklaces, it is natural to seek a more practical combinatrorial criterion for flag positroids.

\begin{quest}[\cite{Oh11}]
    Find a concise combinatorial criterion of flag positroids.
\end{quest}

This question was first posed in \cite[Section 7]{Oh11}, prompting numerous attempts to address it \cite{Oh17, Ben19, Ben22, Jon22}. In particular, several necessary conditions are identified for two-step flag positroids, such as \cite[Proposition 6.2]{Oh17} and \cite[Remark 36]{Ben19}. In contrast, finding sufficient conditions appears to be more challenging. To the authors' knowledge, existing sufficient conditions mainly focus on special cases: \cite{Ben22} provides a necessary and sufficient condition for flag lattice path matroids (a subclass of positroids), and \cite{Ben19} gives a sufficient but not necessary condition for elementary quotients (that is, quotients that decrease the rank by at most 1) of uniform matroids.

Our main result gives a necessary and sufficient characterization of all two-step flag positroids that contain a uniform matroid, based on their CW-arrows (a close relative of Grassmann necklaces, first defined by \cite{Oh17}).

\begin{theorem}\label{thm:uniformCW}
    Given integers $0\leq r\leq k<n$. Let $M$ be a positroid of rank $k-r$ on $[n]$, and let $U_{k,n}$ be the uniform matroid of rank $k$. Then $(M,U_{k,n})$ is a flag positroid if and only if the union of any $r+1$ CW-arrows of $M$ has cardinality at least $k+1$.
\end{theorem}

We note that our criterion is verifiable in $O(n^{3})$ time. The proof of \Cref{thm:uniformCW} employs a novel rank analysis based on the results of \cite{mcalmon2020rank}. We also note that both the statement and the proof of \Cref{thm:uniformCW} differ significantly from the partial characterization in \cite{Ben19}.

\begin{remark}\label{rmk:nonlocality}
Consider the case where $k\ll n$. Then each CW-arrow is a clockwise arrow of length no more than $k$ on the circle of perimeter $n$. While each CW-arrow is a ``local structure" on the circle (that is, involving only a few positions close to each other), \Cref{thm:uniformCW} considers a union of multiple CW-arrows that can be arbitrarily far away from each other. Since whether one positroid is a quotient of another is determined by the combined effects of distantly separated local structures, characterizing flag positroids requires \emph{non-local information} (in terms of positions on the circle). This explains why previous attempts at characterizing flag positroids such as \cite[Conjecture 6.3]{Oh17} have been unsuccessful (we provide a simple counterexample for \cite[Conjecture 6.3]{Oh17} in \Cref{cex:Oh conj}). The authors view this insight as a main conceptual contribution of this paper.
\end{remark}

A canonical ``local property'' of flag positroids is the ``Grassmann necklace containment condition'' \cite[Remark 36]{Ben19}. In contrast to \Cref{rmk:nonlocality}, we show that the containment of both Grassmann necklace and Grassmann conecklace \emph{is} sufficient to characterize the quotients among lattice path matroids.

\begin{theorem}\label{thm:LPM}
Let $M$ and $N$ be lattice path matroids. Then $M$ is a quotient of $N$ if and only if the Grassmann necklace of $N$ contains the Grassmann necklace of $M$ entrywise, and the Grassmann conecklace of $N$ contains the Grassmann conecklace of $M$ entrywise.
\end{theorem}

\Cref{thm:LPM} demonstrates that certain ``local properties'' may naturally ``globalize'' for lattice path matroids but not for general positroids. In fact, the same counterexample for \cite[Conjecture 6.3]{Oh17}, as shown in \Cref{cex:Oh conj}, also shows that \Cref{thm:LPM} is not true for general positroids. 

We remark that a complete characterization of flag lattice path matroids was already obtained by \cite{Ben22}, and the purpose of our \Cref{thm:LPM} lies more in providing contrast to \Cref{thm:uniformCW}.  

Since CW-arrows are closely related to Grassmann necklaces, it is natural to consider counterparts of Grassmann necklace containment conditions in CW-arrows. Surprisingly, they turn out to be intimately related to the \emph{cyclic shift} operation introduced by \cite{Ben19} for decorated permutations (another canonical representation of positroids). Using this connection, we resolve the following conjecture in \cite{Ben19}.

\begin{theorem}[{\cite[Conjecture 37]{Ben19}}]\label{thm:cyclicshift}
If $M,N$ are positroids and $M$ is an elementary quotient of $N$, then the decorated permutations associated with $M$ and $N$ are related by a cyclic shift.
\end{theorem}

Note that since cyclic shift is a counterpart of the necklace containment condition (as shown in \Cref{thm:neckcontain=shift}), it is by our standards a ``local property,'' and thus does not provide a sufficient characterization of flag positroid (see \Cref{cex:Oh conj}).

This paper is structured as follows. \Cref{sec:prelim} provides necessary background on flag matroids and positroids. \Cref{sec:3} builds a connection between the containment of Grassmann necklaces and cyclic shifts, culminating in a proof of \Cref{thm:cyclicshift}. \Cref{thm:uniformCW,thm:LPM} are proved and discussed in further detail in \Cref{sec:4,sec:5}, respectively.

\section{Preliminaries}\label{sec:prelim}

Throughout this work, we denote the set $\{1,2,\dots,n\}$ by $[n]$, and denote the collection of $k$-element subsets of $[n]$ by $\binom{[n]}{k}$.

\subsection{Matroids and quotients of matroids}\label{subsec:quotient}

Matroids are combinatorial structures that abstract and generalize the concept of linear independence. We refer the readers to \cite{Ox11} for a more detailed exposition.

\begin{defin}
    Let $E$ be a finite set and $\mathcal{B}$ be a nonempty collection of subsets in $E$. The pair $M=(E, \mathcal{B})$ is a \emph{matroid} if for all $B, B' \in \mathcal{B}$ and $x \in B\setminus B'$, there exists $y \in B'\setminus B$ such that $(B \cup \left\{y\right\}) \setminus \left\{x\right\} \in \mathcal{B}$. The set $E$ is the \emph{ground set} of $M$, and the elements of $\mathcal{B}=\mathcal{B}(M)$ are the \emph{bases} of $M$.
\end{defin}

\begin{defin}
    It can be shown that all bases of a matroid have the same cardinality. We call this cardinality the \emph{rank} of $M$, denoted by $\rk(M)$. Moreover, we associate with every matroid $M=(E,\mathcal{B})$ a \emph{rank function} $\rk_{M}:2^{[n]}\rightarrow\mathbb{N}$ defined by $\rk_{M}(S):=\max\{|S\cap B|:B\in\mathcal{B}\}$. 
\end{defin}

\begin{prop}[{\cite[Lemma 1.3.1]{Ox11}}]\label{prop:submodular}
For any matroid $M=(E,\mathcal{B})$, its rank function is \emph{submodular}, which means that if $A$ and $B$ are subsets of $E$ then $\rk_{M}(A)+\rk_{M}(B)\geq \rk_{M}(A\cap B)+\rk_{M}(A\cup B)$.
\end{prop}

\begin{defin}
Given a matroid $M = (E, \mathcal{B})$, a subset $I \subseteq E$ is called an \emph{independent set} of $M$ if it is contained in some basis of $M$. Otherwise, we say that $I$ is \emph{dependent}. If a subset $C\subseteq E$ is dependent, but every proper subset of $C$ is independent, then $C$ is said to be a \emph{circuit} of $M$.
\end{defin}

\begin{defin}
    Given a matroid $M = (E, \mathcal{B})$, the collection $\mathcal{B}^* = \left\{ E \setminus B \mid B \in \mathcal{B}\right\}$ also forms the set of bases of a matroid. The matroid $(E,\mathcal{B}^*)$, denoted by $M^*$, is called the \emph{dual} of $M$.
\end{defin}

The rank functions of a matroid and its dual are related by the following formula.

\begin{prop}[{\cite[Proposition 2.1.9]{Ox11}}]\label{prop:rank-dual}
Let $\rk_{M}$ be the rank function of $M$ on the ground set $E$, and $\rk_{M^*}$ be the rank function of $M^{*}$. Then for any $S\subseteq E$ we have
    \[\rk_{M^*}(S) = \rk_{M}(E \setminus S) + |S| - \rk(M).\]
\end{prop}

We now define the quotient relation on matroids, the central object of study in this paper.

\begin{defin}
Given two matroids $M$ and $N$ on the same ground set $E$, we say that $M$ is a \emph{quotient} of $N$, or $(M,N)$ forms a \emph{flag matroid}, if every circuit of $N$ is the union of a collection of circuits of $M$. If in addition the rank of $M$ is exactly one less than the rank of $N$, then $M$ is called an \emph{elementary quotient} of $N$.
\end{defin}

The quotient relation has a convenient characterization using rank functions.

\begin{prop}[{\cite[Proposition 8.1.6]{white1986theory}}]\label{prop:rank-quotient}
Given two matroids $M$, $M'$ on the same ground set $E$, $M$ is a quotient of $M'$ if and only if for all pairs of subsets $A,B$ of $E$ with $A\subseteq B$, 
$$\rk_{M}(B)-\rk_{M}(A)\leq \rk_{M'}(B)-\rk_{M'}(A).$$
\end{prop}

\subsection{Gale orders}

Positroids are a special class of matroids linked with the cyclic structure $(1,2,\dots,n,1)$ of the ground set $[n]$. We first introduce some useful notions for any ordering of the ground set.

\begin{defin}
For any total order $<_{w}$ on $[n]$, the \emph{Gale order} $\leq_{G,w}$ induced by $<_{w}$ is a partial order on subsets of $[n]$: for two $k$-element subsets $A,B\in\binom{[n]}{k}$, we say $A\leq_{G,w}B$ if $a_{i}\leq_{w}b_{i}$ for all $i\in[k]$, where $a_{i}$ (resp. $b_{i}$) is the $i$-th smallest element of $A$ (resp. $B$) under the order $<_{w}$.    
\end{defin}

The special orderings we use are the cyclic orders $<_{i}$ defined by
$$i<_{i}i+1<_{i}\dots<_{i}n<_{i}1<_{i}2<_{i}\dots <_{i}i-1.$$
With a slight abuse of notation, we use $\leq_{i}$ to also denote the Gale order induced by $<_{i}$.

The following proposition provides another characterization of matroid quotients.

\begin{prop}[{\cite[Theorems 1.3.1 and 1.7.1]{Bor03}}]\label{prop:quotient-implies-containment}
If $\mathcal{B}$ is the collection of bases of a matroid on a ground set ordered by $<_{w}$, then there is a unique basis $A\in\mathcal{B}$ such that $A\leq_{G,w}A'$ for any $A'\in\mathcal{B}$. Moreover, if $\mathcal{B}$ and $\mathcal{B}'$ are the collections of bases of two matroids $M$ and $M'$ respectively, and $M$ is a quotient of $M'$, then their unique minimal bases satisfy the containment relation
$$\min_{\leq_{G,w}}\mathcal{B}\subseteq \min_{\leq_{G,w}}\mathcal{B}'.$$
\end{prop}

\subsection{Grassmann necklaces and positroids}

The following combinatorial structure helps us define positroids.
\begin{defin}[\cite{Pos06}, Definition 16.1]
For $0\leq k\leq n$, a \emph{Grassmann necklace of type $(k,n)$}, or simply a \emph{necklace}, is a sequence $I=(I_{1},I_{2},\dots,I_{n})$ of subsets $I_{i}\in\binom{[n]}{k}$ such that for every $i\in [n]$,
\begin{enumerate}
\item if $i\in I_{i}$, then $I_{i+1}=(I_{i}\setminus\{i\})\cup \{j\}$ for some $j\in[n]$,
\item if $i\not\in I_{i}$, then $I_{i+1}=I_{i}$,
\end{enumerate}
where $I_{n+1}:=I_{1}$.
\end{defin}

The original definition of positroids \cite{Pos06} is matroids whose bases correspond to nonzero maximal minors of a $k\times n$ matrix with all maximal minors nonnegative. In this paper, we adopt the following equivalent definition proved by \cite{Oh11}.

\begin{defin}[{\cite[Theorem 6]{Oh11}}]\label{prop:necklace-minimizes}
For any Grassmann necklace $I=(I_{1},\dots,I_{n})$ of type $(k,n)$, the set
$$\mathcal{B}(I)=\left\{B\in\binom{[n]}{k}:I_{i}\leq_{i} B\text{ for all }i\in [n]\right\}$$
forms the collection of bases of some matroid. Such a matroid is said to be a \emph{positroid} on the ground set $[n]$. In addition, we always have $I_{i}=\min_{\leq_{i}}\left\{\mathcal{B}(I)\right\}$ for all $i\in[n]$.
\end{defin}

We also introduce the following dual version of Grassmann necklaces, which coincides with the ``upper Grassmann necklace'' defined in \cite{Oh11}.

\begin{defin}\label{def:conecklace}
    For a positroid $M=([n],\mathcal{B})$, the sequence $(J_{1},\dots,J_{n})$ with $J_{i}:=\max_{\leq i}\{\mathcal{B}(I)\}$ for $i\in[n]$ is called the \emph{Grassmann conecklace}, or simply the \emph{conecklace}, of $M$.
\end{defin}

    As a general convention, in the examples throughout this paper, a set of Arabic numerals is denoted by their concatenation.
    
\begin{eg}\label{eg:grNeck}
    Consider the positroid $P = ([5], \left\{ 1234,1235,1245,1345\right\})$. Its associated Grassmann necklace is $I = (1234, 2341, 3451, 4512, 5123)$, and its Grassmann conecklace is $J=(1345,3451,\allowbreak 4512,5123,1234)$.
\end{eg}

A special class of positroids we study in \Cref{thm:LPM} is \emph{lattice path matroids}, defined as follows. It is proved in \cite{Oh11} that lattice path matroids are positroids.
\begin{defin}[{\cite[Definition 3.1]{Bon03}}]\label{def:LPM}
    A \emph{lattice path matroid (LPM)} $M[U,L]$, where $U,L \in \binom{[n]}{k}$ and $U \leq_{1} L$, is a matroid $([n],\mathcal{B})$ with $\mathcal{B} := \left\{B \in \binom{[n]}{k} \mid U \leq_{1} B \leq_{1} L\right\}$.
\end{defin}

\subsection{Decorated permutations}

The paper \cite{Pos06} also introduces decorated permutations, a class of combinatorial objects that are in bijection with Grassmann necklaces.

\begin{defin}[{\cite[Definition 13.3]{Pos06}}]
    A \emph{decorated permutation} $\pi^:$ on $[n]$ consists of data $(\pi,\col)$, where $\pi \in \mathfrak{S}_{n}$ and $\col:[n] \to \{0,\pm 1\}$ is a mapping such that $\col^{-1}(0)$ is the set of unfixed points of $\pi$. 
\end{defin}

Following the convention of \cite{Pos06}, we add underlines to elements in $\col^{-1}(-1)$, overlines to $\col^{-1}(1)$, and nothing to others. We say elements in $\col^{-1}(1)$ and $\col^{-1}(-1)$ are \emph{loops} and \emph{coloops} respectively.

As a general convention in examples throughout this paper, a decorated permutation $\pi^{:}$ is represented by concatenating the numbers $\pi^{:}(1),\dots,\pi^{:}(n)$, with each number either underlined, overlined, or undecorated. 

\begin{eg}\label{eg:dp}
    Given decorated permutation $\pi^:=(\pi,\col)=41\overline{3}562\underline{7}$, we have $\col^{-1}(0)=4156$, $\col^{-1}(1) = 3$, and $\col^{-1}(-1) = 7$. Thus, its loop is $3$ and coloop is $7$.
\end{eg}

To state the bijection between Grassmann necklaces and decorated permutations on the same ground set from \cite[Lemma 16.2]{Pos06}, we introduce the concept of \emph{anti-exceedance}.

\begin{defin}[\cite{Pos06}]
    Given a decorated permutation $\pi^:$, the set of its \emph{$i$-anti-exceedances} is
    \[W_{i}(\pi^:) := \left\{j \in [n] \mid j <_{i} \pi^{-1}(j) \text{ or } \col(j) = -1\right\}.\]
\end{defin}


\begin{prop}[{\cite[Lemma 16.2]{Pos06}}]\label{prop:decpermtoneck}
    Given a decorated permutation $\pi^{:}$ on $[n]$, the sequence $(W_{1}(\pi^{:}),\dots,W_{n}(\pi^{:}))$ is a Grassmann necklace. Conversely, the following procedure maps a Grassmann necklace $I=(I_{1},\dots,I_{n})$ to a decorated permutation $\sigma^: = (\sigma, \col)$.
    \begin{enumerate}
        \item If $I_{i+1} = (I_{i} \setminus \{i\}) \cup \{j\}$ with $i \neq j$, let $\sigma(i) = j$ and $\col(i) = 0$.
        \item If $I_{i+1} = I_i$, let $\sigma(i) = i$; moreover, if $i \in I_i$, let $\col(i) = -1$, otherwise $\col(i) = 1$.
    \end{enumerate}

    These two maps between Grassmann necklaces and decorated permutations are mutually inverse.
\end{prop}

\begin{eg}
    Consider $\pi^: = \underline{1}5234$. Denote its corresponding necklace by $I$. By calculating anti-exceedances, we obtain the same Grassmann necklace as in \Cref{eg:grNeck}. Applying the procedure in \Cref{prop:decpermtoneck} to this necklace gives rise to $\underline{1}5234$, the decorated permutation we started with.
\end{eg}

Decorated permutations also give a canonical bijection between the Grassmann necklace and the Grassmann conecklace of a positroid.

\begin{prop}[{\cite[Lemma 17]{Oh11}}]\label{prop:coneck deco}
Let $M$ be a positroid over $[n]$. Let $I=(I_{1},\dots,I_{n})$ be its Grassmann necklace and let $\pi^{:}$ be its decorated permutation. Then the Grassmann conecklace $J=(J_{1},\dots,J_{n})$ of $M$ is given by $J_{i}=\pi^{-1}(I_{i})$.
\end{prop}

\begin{remark}\label{remark:different dec convention}
    In some other literature, for example, in \cite{Ard13} and \cite{Ben19}, they use a slightly different convention for decorated permutations compared to \Cref{prop:decpermtoneck}. Specifically, their decorated permutation $ (\pi, \col) $ corresponds to our $ (\pi^{-1}, \col) $, meaning that their ``numbers" are our ``positions" and vice versa.
\end{remark}

\section{Cyclic shifts and necklace containment}\label{sec:3}

In \cite{Ben19}, the authors introduce the notion of \emph{cyclic shifts}, an operation on decorated permutations that cyclically shifts some elements. Their main conjecture, \cite[Conjecture 37]{Ben19} (that is, our \Cref{thm:cyclicshift}), claims that the decorated permutations of any elementary quotient of a positroid $M$ can be represented by a cyclic shift of the decorated permutation of $M$. The primary goal of this section is to prove this conjecture.

We first introduce the definition of cyclic shifts.

\begin{defin}[{\cite[Definition 22]{Ben19}}] \label{def:shift}
    Given a decorated permutation $\pi^{:} = (\pi,\col_{\pi})$ on $[n]$ and a subset $A \subseteq [n]$, we define a new decorated permutation $(\sigma,\col_{\sigma})=\overrightarrow{\rho_{A}}(\pi^{:})$ as follows.
    \begin{enumerate}
        \item For all $i\in A$, let $\sigma(i)=\pi(i)$ and $\col_{\sigma}(i)=\col_{\pi}(i)$.

        \item For $i \notin A$, let $\sigma(i)=\pi(j)$, where $j$ is the maximum element of $[n]\setminus A$ under $<_{i}$.

        \item For $i\notin A$, let $\col_{\sigma}(i)=1$ if $\sigma(i)=i$. Otherwise, define $\col_{\sigma}(i)=0$.
    \end{enumerate}
    We call $\overrightarrow{\rho_{A}}(\pi^{:})$ the \emph{cyclic shift} of $\pi^{:}$ with respect to the set $A$.
\end{defin}

Intuitively, $\overrightarrow{\rho_{A}}$ freezes all positions in $A$. Then, among the remaining positions, it cyclically shifts the numbers one place to the right and decorates new fixed points as loops.

\begin{eg}
    Let $\pi^{:} = \overline{1}65\overline{4}23\underline{7}$, then $\overrightarrow{\rho_{247}}(\pi^{:}) = 361\overline{45}2\underline{7}$.  
\end{eg}

\begin{remark}
    The convention in \cite{Ben19} is slightly different from ours. Specifically, the subset $A$ in our definition represents the fixed positions, while in \cite{Ben19} it represents fixed numbers. This is because we use a different convention of decorated permutations. See \Cref{remark:different dec convention}.
\end{remark}

We are now ready to state our \Cref{thm:cyclicshift} more formally. In addition to the claim that the decorated permutations of an elementary quotient pair is related by a cyclic shift, we can use Grassmann conecklaces (\Cref{def:conecklace}) to characterize the shifted positions.

\begin{theorem}[Formal statement of \Cref{thm:cyclicshift}]\label{thm:shift}
    Let $(M,N)$ be a flag positroid over $[n]$ with $\rk(M)=\rk(N)-1$.
    If $\sigma^{:}$ and $\pi^{:}$ are the decorated permutations of $M$ and $N$ respectively, 
    then $\sigma^{:} = \overrightarrow{\rho_{A}}(\pi^{:})$, where $A = [n] \setminus \bigcup_{i \in [n]}(J^{\pi}_{i} \setminus J^{\sigma}_{i})$. Here  $(J^{\sigma}_{i})_{i\in [n]}$ and $(J^{\pi}_{i})_{i\in [n]}$ are the Grassmann conecklaces of $M$ and $N$, respectively.
\end{theorem}

\begin{notation}
If a positroid $M$ is an elementary quotient of a positroid $N$, and $\sigma^{:}$ and $\pi^{:}$ are their corresponding decorated permutations respectively, then we denote $\sigma^{:} \lessdot \pi^{:}$.
\end{notation}

\begin{remark}\label{rmk:open-question}
It remains an open question whether there is a concise characterization\textemdash based solely on $\pi^{:}$\textemdash of those sets $A$ for which $\overrightarrow{\rho_{A}}(\pi^{:})\lessdot \pi^{:}$.
\end{remark}

\subsection{Relating cyclic shifts to necklace containment}

As we alluded to in the introduction, the cyclic shift operation of decorated permutations turns out to be the ``counterpart'' of containment relations in Grassmann necklaces. The main goal of this subsection is to formally establish this connection. 

To achieve this, we introduce a key tool called \emph{Grassmann matrices}, which act as a bridge between Grassmann necklaces and what we refer to as \emph{Grassmann intervals}. 

Throughout this paper, we adopt the convention that the cyclic interval $(i,i]=\emptyset$.

\begin{defin}\label{def:grassmann-matrix}
Given a decorated permutation $\pi^{:}$ on $[n]$, we define the \emph{Grassmann interval} $S_i^\pi$ associated with each $i \in [n]$ to be the cyclic interval $(\pi^{-1}(i), i]$. In the exceptional case where $\pi^:(i) = \underline{i}$, we set $S_i^\pi \coloneq [n]$.
The \emph{Grassmann matrix} $M^\pi$ is the $n \times n$ binary matrix where the $i$-th row is the indicator vector of the Grassmann interval $S_i^\pi$. Specifically, the entry $(M^\pi)_{i,j}$ is equal to $1$ if $j \in S_i^\pi$ and $0$ otherwise.
\end{defin}

We remark that the ``Grassmann interval'' used here is almost the same as the ``CW-arrow'' defined by \cite{Oh17} (see \Cref{rmk:interval-to-arrow} for further discussion). Our definition of Grassmann intervals serves to demonstrate more clearly the close relation of CW-arrows to Grassmann necklaces, as captured by the following lemma. 

In the following, for decorated permutations $\pi^{:}$ and $\sigma^{:}$ over $[n]$, we use $I^{\pi}$ and $I^{\sigma}$ to denote the corresponding Grassmann necklaces, and use $J^{\pi}$ and $J^{\sigma}$ to denote the corresponding Grassmann conecklaces, respectively. 

\begin{lemma}\label{lemma:grassmann matrix}
    Given a Grassmann matrix $M^{\pi}$, the $j$-th column of $M^{\pi}$ is the indicator vector for $I^{\pi}_j$. Consequently, the sum of each column of $M^{\pi}$, which we hereafter denote by $\rk(\pi^{:})$, coincides with $|I^{\pi}_{1}|$, the rank of the associated positroid.
\end{lemma}

\begin{proof}
    By definition, \[(M^\pi)_{i,j}=\begin{cases}
 1, &\text{if }i\text{ is a }j\text{ anti-exceedance} \ (i<_{j}\pi^{-1}(i)\text{ or }\col_\pi(i)=-1);\\
  0, &\text{otherwise.}
 \end{cases}\]
 And hence by \Cref{prop:decpermtoneck}, the $j$-th column of $M^{\sigma}$ is the indicator vector for $I^{\sigma}_j$, and the sum of each column of $M^{\sigma}$ equals $\rk(\sigma^{:}).$
\end{proof}
\begin{eg}
     Let $\pi^{:} = \overline{1}65\overline{4}23\underline{7}$, then the Grassmann matrix is
     \[M^\pi=\begin{pmatrix}
         0&0&0&0&0&0&0\\
         1&1&0&0&0&1&1\\
         1&1&1&0&0&0&1\\
         0&0&0&0&0&0&0\\
         0&0&0&1&1&0&0\\
         0&0&1&1&1&1&0\\
         1&1&1&1&1&1&1
     \end{pmatrix}.\]
     The Grassmann interval $S_3^\pi=(\pi^{-1}(3),3]=(6,3]=\{7,1,2,3\}=\{1,2,3,7\}$ can be read through the third row of $M^\pi$, and $I^{\pi}_4=\{4,5,6\}$ of the Grassmann necklace $I^{\pi}$ can be read through the fourth column of $M^\pi$.
\end{eg}

\Cref{lemma:grassmann matrix} immediately implies the following corollary.
\begin{cor}\label{cor:neckcontain=vecdom}
    $I^{\sigma}_{i}\subseteq I^{\pi}_{i}$ for any $i\in [n]$ if and only if $S^{\sigma}_i\subseteq S^{\pi}_i$ for any $i\in[n]$
\end{cor}

\begin{proof}
By \Cref{lemma:grassmann matrix} and \Cref{def:grassmann-matrix} respectively, both conditions are equivalent to $(M^{\sigma})_{i,j}\le (M^{\pi})_{i,j}$ for any $i,j\in [n]$.
\end{proof}

To establish connections with the cyclic shift $\sigma^{:} = \overrightarrow{\rho_{A}}(\pi^{:})$, we define \emph{shift intervals} $S^{\pis}_i$. In fact, shift intervals record the rows of $M^\pi-M^\sigma$.

\begin{defin}\label{def:vec_w}
Given two decorated permutations $\sigma^:, \pi^:$ over $[n]$, 
we define the \emph{$i$-th shift interval} as cyclic interval $(\pi^{-1}(i), \sigma^{-1}(i)]$. In the exceptional case where $\sigma^:(i)=\overline{i}$ and $\pi^:(i)=\underline{i}$, we set $S^{\pis}_i\coloneq [n]$.
\end{defin}

\begin{eg}
    Let $\pi^:=456123$ and $\sigma^:=2461\overline{5}3$. We have
    \begin{align*}
         &S^{\pi}_2=(\pi^{-1}(2),2]=\{6,1,2\}, \quad S^{\sigma}_2=(\sigma^{-1}(2),2]=\{2\}, \text{ and}\\
        &S^\pis_2=(\pi^{-1}(2),\sigma^{-1}(2)]=\{6,1\}=S^{\pi}\setminus S^{\sigma}.
    \end{align*}
\end{eg}

We use shift intervals as an intermediary to bridge Grassmann necklace containment and the cyclic shift. In the next two lemmas, we separately show that both the necklace containment relation and the cyclic shift condition are equivalent to the same condition on shift intervals, thereby establishing the desired equivalence between necklace containment and the cyclic shift.

\begin{lemma}\label{lemma:Gr contain and shift}
     Given two decorated permutations $\sigma^:, \pi^: $ over $[n]$ where $\rk(\sigma^:)=\rk(\pi^:)-1$, $\bigsqcup_{i=1}^{n}S^\pis_i=[n]$ if and only if $S^{\sigma}_i\subseteq S^{\pi}_i$ for any $i\in[n]$. 
 \end{lemma}
\begin{proof}
    \textbf{The ``if'' direction:}
by \Cref{lemma:grassmann matrix}, as multi-sets, $\bigsqcup_{i=1}^n S^{\pi}_i$ is $\rk(\pi^:)$ copies of $[n]$, and $\bigsqcup_{i=1}^n S^{\sigma}_i$ is $\rk(\sigma^:)$ copies of $[n]$. Hence, if $S^{\sigma}_i\subseteq S^{\pi}_i$ for any $i\in[n]$, then we have
\[[n]=\rk(\pi^:)[n]-\rk(\sigma^:)[n]=\left(\bigsqcup_{i=1}^n S^{\pi}_i\right)\setminus \left(\bigsqcup_{i=1}^n S^{\sigma}_i\right)=\bigsqcup_{i=1}^n \left(S^{\pi}_i\setminus S^{\sigma}_i\right)=\bigsqcup_{i=1}^n S^\pis_i.
\]

\textbf{The ``only if'' direction:} if $\bigsqcup_{i=1}^{n}S^\pis_i=[n]$, then by the triangle inequality,
\[n=\sum_{i=1}^n|S^\pis_i| \geq \sum_{i=1}^n\left(|S^{\pi}_i|-|S^{\sigma}_i|\right)=\sum_{i=1}^n|S^{\pi}_i|-\sum_{i=1}^n|S^{\sigma}_i|=n\left(\rk(\pi^:)-\rk(\sigma^:)\right)=n.\]
Equality holds only if $S^{\sigma}_i\subseteq S^{\pi}_i$ for any $i\in[n]$.
\end{proof}

\begin{lemma}\label{lem:vec=shift}
Given two decorated permutations $\sigma^:, \pi^:$ over $[n]$ where $\rk(\sigma^:)=\rk(\pi^:)-1$, $\sigma^{:}=\overrightarrow{\rho_{A}}(\pi^{:})$ for some $A\subseteq [n]$ if and only if $\bigsqcup_{i=1}^{n}S^\pis_i=[n]$.
\end{lemma}
\begin{proof}
\textbf{The ``only if'' direction:} we assume that $\sigma^{:}= \overrightarrow{\rho_{A}}(\pi^{:})$ for some $A\subseteq [n]$. If $|A|=n$ then $\sigma^{:}=\pi^{:}$, contradicting the assumption $\rk(\sigma^{:})=\rk(\pi^{:})-1$. If $|A|=n-1$, then $\sigma^{:}\neq \pi^{:}$ implies the shifted element turns from a coloop to a loop. Let $i$ be the only shifted element, then $S^\pis_{i}=[n]$ and $S^\pis_{j}=\emptyset$ for any $j\in [n]\setminus \left\{i\right\}$, and consequently statement $\bigsqcup_{i=1}^{n}S^\pis_i=[n]$ holds.

Now assume $|A|\leq n-2$. Let $[n]\setminus A=\left\{i_{1},\dots,i_{\ell}\right\}$, where $i_{1}<_{1}\dots<_{1}i_{\ell}$. Denote $i_{\ell+1}=i_{1}$. Then we know from \Cref{def:shift} that $\pi([n]\setminus A)$ are the shifted numbers. For each $j\not\in \pi([n]\setminus A)$, by \Cref{def:vec_w} we have $S^\pis_{j}=\emptyset$. We also know from \Cref{def:shift} that $\sigma^{-1}(\pi(i_{r}))=i_{r+1}$ for each $r\in [\ell]$. So 
\[[n]=\bigsqcup_{r=1}^{\ell}(i_{r},i_{r+1}]=\bigsqcup_{r=1}^{\ell}\left(\pi^{-1}(\pi(i_{r})),\sigma^{-1}(\pi(i_{r}))\right]=\bigsqcup_{r=1}^{\ell}S^\pis_{\pi(i_{r})}=\bigsqcup_{i=1}^{n}S^\pis_i.\]

\textbf{The ``if'' direction:} assume that $\bigsqcup_{i=1}^{n}S^\pis_i=[n]$ holds. Reorder all the non-empty cyclic intervals as $S_{\pi(i_1)}^{\pi,\sigma}, \dots, S_{\pi(i_\ell)}^{\pi,\sigma}$ such that $i_1<_1 \dots<_1 i_\ell$. We claim $\sigma^:=\overrightarrow{\rho_{A}}(\pi^{:})$ for $A\coloneq[n]\setminus\{i_{1},\dots,i_{\ell}\}$. 

Note that the left endpoint of each cyclic interval $S^\pis_{\pi(i_r)}$ is $\pi^{-1}(\pi(i_{r}))=i_r$. By $\bigsqcup_{i=1}^{n}S^\pis_i=[n]$ and $i_1<_1 \dots<_1 i_\ell$, we have $S^\pis_{\pi(i_r)}=(i_r, i_{r+1}]$. Moreover, by \Cref{def:vec_w}, $S^\pis_{\pi(i_r)}=(\pi^{-1}(\pi(i_r)), \sigma^{-1}(\pi(i_r))]$, and hence $i_{r+1}=\sigma^{-1}(\pi(i_r))$ for any $r\in [\ell]$.

The shifted numbers of $\overrightarrow{\rho_{A}}(\pi^{:})$ are $\pi([n]\setminus A)$, corresponding to $\pi(i_1), \pi(i_2),\dots, \pi(i_{\ell})$. We have shown that $\sigma(i_{r+1})=\pi(i_r)$ for any $r\in [\ell]$. Moreover, as \Cref{lemma:Gr contain and shift} shows each $S^{\sigma}_i$ is contained in $S^{\pi}_i$, which would be violated if any new fixed point of $\sigma^:$ were a coloop.
By \Cref{def:shift}, we conclude that $\sigma^:=\overrightarrow{\rho_{A}}(\pi^{:})$, as claimed.
\end{proof}

We now conclude with the following theorem.

\begin{theorem}\label{thm:neckcontain=shift}
    Let $M$ and $N$ be two positroids on $[n]$, and $\rk(M)=\rk(N)-1$. Let $\sigma^{:}$ and $\pi^{:}$ be the decorated permutations of $M$ and $N$, and $I^{\sigma}$ and $I^{\pi}$ be the Grassmann necklaces of $M$ and $N$, respectively. Then, $I^{\sigma}_{i}\subseteq I^{\pi}_{i}$ for all $i\in [n]$ if and only if there exists $A\subseteq [n]$ such that $\sigma^{:} = \overrightarrow{\rho_{A}}(\pi^{:})$.
\end{theorem}

\begin{proof}
The conclusion follows directly by combining \Cref{lemma:Gr contain and shift,lem:vec=shift} and \Cref{cor:neckcontain=vecdom}.
\end{proof}

\subsection{Characterizing shifted positions}

As stated in \Cref{thm:shift}, for elementary quotient pairs, we can characterize the shifted positions using Grassmann conecklaces. We first prove the following lemma that connects Grassmann conecklaces with positions of decorated permutations.

\begin{lemma}\label{A} 
Assume $\sigma^{:}=(\sigma,\col_{\sigma})$ and $\pi^{:}=(\pi,\col_{\pi})$ are decorated permutations on $[n]$, and let $J^{\sigma}$ and $J^{\pi}$ be the Grassmann conecklaces of $\sigma^{:}$ and $\pi^{:}$, respectively. If $J^{\sigma}_{i}\subseteq J^{\pi}_{i}$ for any $i\in [n]$, then we have
\[\left\{a\in [n]\mid \sigma(a)\neq \pi(a)\text{ or }\col_{\sigma}(a)\neq \col_{\pi}(a)\right\}=\bigcup_{i=1}^{n}\left(J^{\pi}_{i}\setminus J^{\sigma}_{i}\right).\]
\end{lemma}

\begin{proof}
    By \Cref{prop:coneck deco}, \Cref{prop:decpermtoneck}, and \Cref{lemma:grassmann matrix}, we have 
    \begin{align*}
        J^{\pi}_i&=\pi^{-1}(I^{\pi}_i)=\left\{\pi^{-1}(a)\in [n] \mid a<_{i}\pi^{-1}(a)\text{ or }\col_{\pi}(a)=-1\right\}\\
        &=\left\{\pi^{-1}(a)\in [n] \mid i\in S^{\pi}_a\right\}=\{a\in[n]\mid i\in S^{\pi}_{\pi(a)}\}.
    \end{align*}
    Therefore, since $J^{\sigma}_{i}\subseteq J^{\pi}_{i}$ for any $i\in [n]$, we have $S^{\sigma}_{\sigma(a)}\subseteq S^{\pi}_{\pi(a)}$ for any $a\in [n]$. Hence, \
\begin{align*}
   J^{\pi}_{i}\setminus J^{\sigma}_{i}&=\left\{a\in[n]\mid i\in S^{\pi}_{\pi(a)}\setminus S^{\sigma}_{\sigma(a)} \right\}\\
   &=\left\{a\in[n]\mid i\in (\sigma(a),\pi(a)]\right\}\cup\left\{a\in [n]\mid\col_{\sigma}(a)=1, \col_{\pi}(a)=-1\right\}.
    \end{align*}
    That is to say, \[\bigcup_{i=1}^{n}\left(J^{\pi}_{i}\setminus J^{\sigma}_{i}\right)=\left\{a\in[n]\mid  (\sigma(a),\pi(a)]\neq \emptyset\right\}\cup\left\{a\in [n]\mid\col_{\sigma}(a)=1, \col_{\pi}(a)=-1\right\},\]
    or equivalently, \[\bigcup_{i=1}^{n}\left(J^{\pi}_{i}\setminus J^{\sigma}_{i}\right)=\left\{a\in [n]\mid \sigma(a)\neq \pi(a)\text{ or }\col_{\sigma}(a)\neq \col_{\pi}(a)\right\},\]as desired.
    \end{proof}

We are now ready to conclude the proof of \Cref{thm:shift}.

\begin{proof}[Proof of \Cref{thm:shift}]\label{proof:thm-shift}
By \Cref{prop:quotient-implies-containment} and \Cref{prop:necklace-minimizes}, if the positroids in \Cref{thm:neckcontain=shift} form a flag positroid $(M,N)$, the Grassmann necklaces satisfy $I^{\sigma}_{i}\subseteq I^{\pi}_{i}$. Therefore, \Cref{thm:neckcontain=shift} implies that there exists $A\subseteq [n]$ such that $\sigma^{:}=\overrightarrow{\rho_{A}}(\pi^{:})$. By the definition of cyclic shift (\Cref{def:shift}), we must have $A=\left\{a\in [n]\mid \sigma(a)=\pi(a)\text{ and }\col_{\sigma}(a)=\col_{\pi}(a)\right\}$, which is equal to $[n]\setminus \bigcup_{i=1}^{n}\left(J^{\pi}_{i}\setminus J^{\sigma}_{i}\right)$ by \Cref{A}.
\end{proof}

\section{Positroid quotients of uniform matroids}\label{sec:4}

The goal of this section is to prove \Cref{thm:uniformCW}, which gives a complete characterization of flag positroids $(M,U_{k,n})$ in terms of CW-arrows. Our proof relies on input from the previous papers \cite{Oh17,mcalmon2020rank}, which is discussed in \Cref{subsec:previous-results}. We then present our contributions and prove \Cref{thm:uniformCW} in \Cref{subsec:our-ideas}. Finally, in \Cref{subsec:elementary-quotients}, we explore the $r=1$ special case of \Cref{thm:uniformCW} and its relation to \Cref{thm:cyclicshift}. 

\subsection{Preparations and previous results}\label{subsec:previous-results}
We first introduce the definition of CW-arrows.
\begin{defin}[\cite{Oh17}]
    Let $\pi^: = (\pi, \col)$ be a decorated permutation on $[n]$. For each $i \in [n]$, the \emph{CW-arrow} starting at $i$ is defined by\[
    C_{i}=\begin{cases}
[n],&\text{if }\pi(i)=i\text{ and }\col(i)=-1,\\
\{j\in [n]:j\leq_{i}\pi(i)\}, &\text{otherwise.}
\end{cases}
    \]
\end{defin}

\begin{remark}\label{rmk:interval-to-arrow}
The intuition behind CW-arrows is that, we put $1, \dots, n$ on a circle in the clockwise order, and draw an arrow from $i$ to $\pi(i)$ clockwise. The numbers covered by the arrow are in the CW-arrow starting at $i$. When $i$ is a coloop, the arrow goes through the entire circle, whereas when $i$ is a loop, its CW-arrow is a singleton. Moreover, CW-arrows can be viewed as Grassmann intervals with their left endpoints included.
\end{remark}

We give the readers some working examples to check \Cref{thm:uniformCW}.

\begin{eg}\label{eg:cw}
    Consider the uniform matroid $U_{4,6}$, a positroid $P$ of rank $2$ whose decorated permutation is $\overline{1}54623$, and a positroid $Q$ of rank $1$ whose decorated permutation is $6 \overline{2345} 1$. The CW-arrows of $P$ are $1, 2345, 34, 456, 5612$, $6123$. The readers could check that the union of any $3$ of them has cardinality of at least $5$. The CW-arrows of $Q$ are $123456,2,3,4,5,61$. If we take $2,3,4,5$, the cardinality of the union of these $4$ CW-arrows is less than $5$. Therefore, $P$ is a quotient of $U_{4,6}$, but $Q$ is not.
\end{eg}

We then define the \emph{CW-function}, which counts the number of CW-arrows contained in a set.

\begin{defin}[\cite{Oh17}]
Let $M$ be a positroid on $[n]$ without coloops and let $\rk_{M}:2^{[n]}\rightarrow \mathbb{N}$ be its rank function. If $(C_{1},\dots,C_{n})$ are the CW-arrows of $M$, then the \emph{CW-function} $\cw_{M}:2^{[n]}\rightarrow \mathbb{N}$ is defined by
    \[
    \cw_{M}(A):=
\begin{cases}
\sum_{i=1}^{n}\mathbbm{1}\{C_{i}\subseteq A\},&\text{if }A\neq [n],\\
n-\rk_{M}([n]),&\text{if }A=[n].
\end{cases}
    \]
\end{defin}

Readers might concern about the assumption that $M$ has no coloops. This is because later in the proof of the theorem, we always only deal with positroids with no coloops. 

\begin{remark}\label{remark:cw le r}
    Given the definition, the condition in \Cref{thm:uniformCW} could be rephrased as ``for any $k$-element subset $A$ of $[n]$ we have $\cw_{M}(A) \le r$.''
\end{remark}

\begin{eg}
    Consider the matroids from \Cref{eg:cw}. We have $\cw_{P}(3456) = 2$ and $\cw_{Q}(1456) = 3$. 
\end{eg}

The following lemma shows that we can read the rank of a cyclic interval from the CW-function.

\begin{lemma}[{\cite[Proposition 3.4]{Oh17}}]\label{lem:rankcyclic}
Let $M$ be a positroid on $[n]$. Let $(I_{1},\dots,I_{n})$ be the Grassmann necklace associated with $M$. If 
$J$ is a cyclic interval of the form $J=\{\ell\in [n]:\ell\leq_{i} j\}$ for some $j\in [n]$, then we have
$$\rk_{M}(J)=|J|-\cw_{M}(J).$$
\end{lemma}

\begin{remark}
The paper \cite{Oh17} did not define the value of the CW-function on the whole set $A=[n]$, and their Proposition 3.4 does not include this case. Nevertheless, in our \Cref{lem:rankcyclic}, the case $J=[n]$ is trivial given that we defined $\cw_{M}([n])$ to be $n-\rk_{M}([n])$.
\end{remark}

We now define a dual version of CW-arrows, which arise naturally by considering counter-clockwise arrows instead of clockwise arrows (see \Cref{rmk:interval-to-arrow}).

\begin{defin}[\cite{Oh17}]\label{def:CCW-arrows}
Let $\pi^{:}=(\pi,\col)$ be a decorated permutation. For each $i\in [n]$, the CCW-arrow starting at $i$ is defined by
\[CC_{i}=\begin{cases}
[n],&\text{if }\pi(i)=i\text{ and }\col(i)=1,\\
\{j\in [n]:j\leq_{\pi(i)}i\}, &\text{otherwise.}
\end{cases}\]
\end{defin}

The following function counts the number of CCW-arrows contained in a set. Note that here we assume that $M$ has no loops rather than $M$ has no coloops. The reason is that although we primarily deal with positroids without coloops, we only consider CCW-arrows of the matroid dual of such positroids.

\begin{defin}
Let $M$ be a positroid on $[n]$ without loops. If $(CC_{1},\dots,CC_{n})$ are the CCW-arrows of $M$, then the \emph{CCW-function} $\ccw_{M}:2^{[n]}\rightarrow \mathbb{N}$ is defined by
\[\ccw_{M}(A):=
\begin{cases}
\sum_{i=1}^{n}\mathbbm{1}\{CC_{i}\subseteq A\},&\text{if }A\neq [n],\\
\rk_{M}([n]),&\text{if }A=[n].
\end{cases}
\]
\end{defin}

The following key result from \cite{mcalmon2020rank} shows that the rank of an arbitrary subset of $[n]$\textemdash not just cyclic intervals as in \Cref{lem:rankcyclic}\textemdash can be read from the CCW-arrows, though in a less efficient manner.

\begin{lemma}[{\cite[Theorem 25]{mcalmon2020rank}}]\label{lem:main}
Let $M$ be a positroid on $[n]$ without loops and let $A\subseteq [n]$. There exists a partition $A=A_{1}\sqcup\dots\sqcup A_{t}$ such that
\begin{equation}\label{eq:MOrank}
\rk_{M}(A)=\sum_{j=1}^{t}\Big(\rk_{M}([n])-\ccw_{M}([n]\setminus A_{j})\Big).
\end{equation}
\end{lemma}

\begin{remark}
The paper \cite{mcalmon2020rank} establishes their results for positroids without loops \emph{and} without coloops, for simplicity (\cite[Remark 11]{mcalmon2020rank}). However, to justify \Cref{lem:main}, we may assume without loss of generality that $M$ has no coloop so that \cite[Theorem 25]{mcalmon2020rank} still applies. Indeed, in \eqref{eq:MOrank}, deleting a coloop element contained in $A$ decreases both sides of the equation by 1, while deleting a coloop element contained in $[n]\setminus A$ leaves both sides unchanged.  
\end{remark}

\subsection{Our ideas}\label{subsec:our-ideas}

The starting point of our investigation is the question of whether positroid quotients can be characterized concisely using rank functions. However, there are two immediate obstacles to this approach: (1) although rank functions do characterize general matroid quotients (\Cref{prop:rank-quotient}), verifying this criterion requires exponential time; and (2) even for a given set in a positroid, computing its rank is not straightforward (\Cref{lem:main}). Our key idea is that when one of the two positroids is a uniform matroid, \Cref{prop:rank-quotient} admits a simplification that (when sufficiently exploited) may allow us to bypass the need to compute ranks directly, thereby avoiding the computational complexity lying in (1) and (2).

Concretely, we observe the following version of \Cref{prop:rank-quotient} for quotients of uniform matroids.

\begin{lemma}\label{lem:rank_char_of_quotient}
Let $M$ be a matroid on $[n]$ of rank $k-r$. Then $M$ is a quotient of $U_{k,n}$ if and only if $\rk_{M}(A)=k-r$ for any $k$-element subset $A\subseteq [n]$.
\end{lemma}

\begin{proof}
    \textbf{The ``only if'' direction:} take arbitrary $A \subseteq [n]$ where $|A| = k$. Since $M$ has rank $k-r$, we know that $\rk_{M}(A) \leq k-r$ from the definition. Then, based on the rank inequality in \Cref{prop:rank-quotient}, we know that \[
    \rk_{M}([n]) - \rk_{M}(A) \leq \rk_{U_{k,n}}([n]) - \rk_{U_{k,n}}(A) \implies \rk_{M}(A) \geq \rk_{M}([n]),
    \] because $\rk_{U_{k,n}}([n]) = \rk_{U_{k,n}}(A) = k$. Moreover, as $\rk_{M}([n]) = k-r$, we have $\rk_{M}(A) \geq k-r$. Therefore, $\rk_{M}(A) = k-r$. 
    
    \textbf{The ``if'' direction:} consider arbitrary $B \subseteq A \subseteq [n]$. By \Cref{prop:rank-quotient}, we only need to verify the rank inequality
    \[
    \rk_{M}(A)-\rk_{M}(B)\leq \rk_{U_{k,n}}(A)-\rk_{U_{k,n}}(B).
    \]
    First, notice that 
    \[\rk_{U_{k,n}}(A) =\begin{cases}
k,&\text{if }|A| \ge k,\\
|A|, &\text{otherwise.}
\end{cases}\] When $|A|, |B| \ge k$, the rank inequality holds immediately since we have zeros on both sides. When $|A|, |B| < k$, because $B \subseteq A$, we know that \[
    \rk_{M}(A) - \rk_{M}(B) \le |A| - |B| = \rk_{U_{k,n}}(A) - \rk_{U_{k,n}}(B).
    \] Finally, consider the case $|A| \ge k$ and $|B| < k$. Then it suffices to show
    \[\rk_{M}(B) \ge |B| - r,\]
    as $\rk_{U_{k,n}}(A)=k$, $\rk_{M}(A)=k-r$, and $\rk_{M}(B)=|B|$.
    By our assumption, if we take a set $C$ such that $B\subseteq C$ and $|C|=k$, we have \[\rk_{M}(B) + k - |B| \geq \rk_{M}(C)= k-r \implies  \rk_{M}(B) \ge |B| - r, \] which gives the desired inequality. 
\end{proof}

As noted in \Cref{remark:cw le r}, our goal is to show that the criterion $\cw_{M}(A)\leq r$ for all $k$-element subsets $A$ of $[n]$ is equivalent to the quotient relation. The forward implication is the more challenging direction, which requires proving that the CW-function upper bounds on $k$-element subsets are sufficient to guarantee the quotient property. Towards that end, we first observe that the CW-function upper bounds on $k$-element subsets naturally ``lift'' to upper bounds for larger subsets, as formalized by the following lemma.

\begin{lemma}\label{lem:cwA-1}
Let $M$ be a positroid on $[n]$ without coloops. For any nonempty subset $A\subseteq [n]$, there is an element $x\in A$ such that $\cw_{M}(A\setminus\{x\})\geq \cw_{M}(A)-1$.
\end{lemma}
\begin{proof}
If $A=[n]$, pick an arbitrary $x\in [n]$. Since $A = [n]$, by definition, $\cw_{M}(A) = n-\rk_{M}([n])= n - |W_1(\pi^:)|$, where $\pi^:$ is the decorated permutation associated with $M$. Notice that $|W_1(\pi^:)|$ equals the number of CW-arrows $C_i$ where $i > \pi(i)$, and for all such $C_i$, we have $1 \in C_i$. Thus, $A \setminus \{ 1 \}$ contains all $n$ CW-arrows except for those $|W_1(\pi^:)|$ CW-arrows and one additional CW-arrow starting at $1$. Therefore, we have \[\cw_{M}(A\setminus\{x\})=n-|W_1(\pi^:)|-1=\cw_{M}(A)-1.\] 

If $A\neq [n]$, pick an element $x\in A$ such that $(x-1 \pmod n)\not\in A$. This element exists; otherwise, $A = [n]$. There is at most one CW-arrow that is contained in $A$ and contains $x$, namely, the one starting at $x$. Therefore, $\cw_{M}(A\setminus\{x\})\geq \cw_{M}(A)-1.$
\end{proof}

\begin{cor}\label{cor:trickle_up}
Let $M$ be a positroid on $[n]$ without coloops. If $\cw_{M}(A)\leq r$ for all $k$-element subsets $A\subseteq [n]$, then $\cw_{M}(A)\leq |A|-k+r$ for all $A\subseteq [n]$ such that $|A|\geq k$.
\end{cor}

\begin{proof}
    Take an arbitrary $A$ where $|A| \ge k$. By \Cref{lem:cwA-1}, we know that there exist elements $x_1, \dots, x_{|A|-k}$ such that $\cw_{M}(A\setminus \{x_1,\dots,x_{|A|-k}\}) \ge \cw_{M}(A) - (|A| - k)$. Combining with the assumption yields the desired inequality. 
\end{proof}

Before heading into the proof of \Cref{thm:uniformCW}, we record another basic fact about the CW-function: it provides a simple upper bound for the rank function of a positroid. This helps prove the reverse (and easier) direction that if a positroid $M$ is a quotient of $U_{k,n}$, then $\cw_M(A)\le r$ for any $k$-element subset $A\subseteq [n]$.

\begin{lemma}\label{lem:rankupperbound}
Let $M$ be a positroid on $[n]$ without coloops. Then for any proper subset $A\subsetneq [n]$, we have $\rk_{M}(A)\leq |A|-\cw_{M}(A)$.
\end{lemma}
\begin{proof}
If $A$ is a cyclic interval, note that $\rk_{M}(A)=|A|-\cw_{M}(A)$ from \Cref{lem:rankcyclic}. If $A$ is not, it can be written as a disjoint union of cyclic intervals $A =\bigsqcup A_i$. Then by submodularity of the rank function (\Cref{prop:submodular}), we have \[
\rk_{M}(A) = \rk_{M}\left( \bigsqcup A_i \right)\le \sum \rk_{M}(A_i) = \sum|A_i| - \sum \cw_{M}(A_i) = |A| - \cw_{M}(A).\qedhere
\]
\end{proof}

Now we are ready to prove our theorem. The ``if'' direction is the harder direction, and the main idea is to apply \Cref{lem:main} to the dual matroid of $M$.

\begin{proof}[Proof of \Cref{thm:uniformCW}]

We first prove the ``only if'' direction by contradiction. Given that $M$ is a quotient of $U_{k,n}$, it follows that $M$ has no coloops since coloops are not contained in any circuits. Assume there exist $r+1$ CW-arrows of $M$ whose union has cardinality less than $k+1$, and notice that this is equivalent to saying that there exists a $k$-element subset $A \subseteq [n]$ where $\cw_{M}(A) \ge r+1$. By \Cref{lem:rankupperbound} we have $\rk_{M}(A) \le k-r-1$, and by \Cref{lem:rank_char_of_quotient} we know that $M$ is not a quotient of $U_{k,n}$, resulting in a contradiction.

We then prove the ``if'' direction. By \Cref{remark:cw le r}, the assumption is equivalent to $\cw_{M}(A)\leq r$ for all $k$-element subset of $[n]$. We firstly argue that $M$ cannot have coloops. Suppose $j$ is a coloop of $M$. Consider the CW-arrows whose end points are in $[j+1,j+k] \setminus W_j(\pi^{:})$, where $\pi^{:}$ is the decorated permutation associated with $M$. By \Cref{prop:decpermtoneck} we know that $W_{j}(\pi^{:})$ has size $k-r$, and since $\pi^{:}(j)=\underline{j}$ we know that $j\in W_{j}(\pi^{:})$. Thus the number of CW-arrows whose end points are in $[j+1,j+k] \setminus W_j(\pi^{:})$ is at least $k-(k-r-1) = r+1$. Since all these CW-arrows are contained in $[j+1,j+k]$ due to the definition of $W_{j}(\pi^{:})$, this implies that $\cw_{M}([j+1,j+k])\geq r+1$, contradicting the assumption. 

Now fix a $k$-element subset $A\subseteq [n]$. Let $M^{*}$ be the dual of $M$. By applying the \Cref{lem:main} to $[n]\setminus A$ and the dual matroid $M^{*}$, we obtain a partition $[n]\setminus A=A_{1}\sqcup\dots\sqcup A_{t}$ such that
\begin{equation}\label{eq:apply-main}
\rk_{M^{*}}([n]\setminus A)=\sum_{j=1}^{t}\Big(\rk_{M^{*}}([n])-\ccw_{M^{*}}([n]\setminus A_{j})\Big).
\end{equation}
We thus have
\begingroup
\allowdisplaybreaks
\begin{align*}
\rk_{M}(A)&=|A|-n+\rk_{M}([n])+\rk_{M^{*}}([n]\setminus A) &(\text{by \Cref{prop:rank-dual}})\\
&=|A|-n+\rk_{M}([n])+\sum_{j=1}^{t}\Big(\rk_{M^{*}}([n])-\ccw_{M^{*}}([n]\setminus A_{j})\Big)&(\text{by \eqref{eq:apply-main}})\\
&=|A|-n+\rk_{M}([n])+\sum_{j=1}^{t}\Big(n-\rk_{M}([n])-\cw_{M}([n]\setminus A_{j})\Big)&(\text{by duality})\\
&\geq |A|-n+\rk_{M}([n])+\sum_{j=1}^{t}\Big(n-\rk(M)-\left|[n]\setminus A_{j}\right|+\rk(M)\Big)&(\text{by \Cref{cor:trickle_up}})\\
&=\rk(M).
\end{align*}
\endgroup
Since $A$ is arbitrary, an application of \Cref{lem:rank_char_of_quotient} shows that $M$ is a quotient of $U_{k,n}$.
\end{proof}

\subsection{The case of elementary quotients}\label{subsec:elementary-quotients}

In this subsection, we interpret \Cref{thm:uniformCW} in the context of \Cref{sec:3}. Recall from \Cref{rmk:open-question} that, given a decorated permutation $\pi^{:}$ over $[n]$, we want to find a concise characterization of sets $A\subseteq [n]$ such that $\overrightarrow{\rho_{A}}(\pi^{:})\lessdot \pi^{:}$. Since \Cref{thm:uniformCW} is a necessary and sufficient condition for the quotient, it helps answer this question for the case $\pi^{:}=\pi_{k,n}$.

For every proper subset $A\subseteq [n]$, there is a unique way to decompose $A$ into a collection of cyclic intervals that are disjoint and pairwise non-adjacent, written as $A=[a_{1},b_{1}]\sqcup \dots \sqcup [a_{s},b_{s}]$ where no pair $i,j\in [s]$ satisfies $b_{j}+1\equiv a_{i}\pmod n$. We call each of these cyclic intervals a \emph{cyclic component} of $A$.

\begin{eg}
    Let $A\subseteq [9]$ be $\{1,2,4,6,7,9\}$, then the cyclic intervals of $A$ are $\{4\},[6,7]$ and $[9,2]$.
\end{eg}

Using the concept of cyclic components, we can state the main result of this subsection.

\begin{theorem}\label{thm:uniform}
    Let $1\leq k\leq n-1$, and let $\sigma^:$ be a decorated permutation on $[n]$. Then $\sigma^{:}\lessdot \pi_{k,n}$ if and only if $\sigma^: = \overrightarrow{\rho_{A}}(\pi_{k,n})$ for some $A \subseteq [n]$, where the union of any two distinct cyclic components of $A$ has cardinality at most $k-1$.
\end{theorem}

\begin{proof}
\textbf{The ``only if'' direction:} the existence of an $A\subseteq[n]$ such that $\sigma^{:}=\overrightarrow{\rho_{A}}(\pi_{k,n})$ follows from \Cref{thm:shift}. Let $[a_{1},b_{1}]$ and $[a_{2},b_{2}]$ be any two distinct cyclic components of $A$. We show that $|[a_{1},b_{1}]|+|[a_{2},b_{2}]|\leq k-1$. Let $c_{1}=\pi_{k,n}(a_{1}-1)=(a_{1}+k-1\pmod n)$. After the cyclic shift, we would have $\sigma(b_{1}+1)=c_{1}$. By  \Cref{lemma:Gr contain and shift} and \Cref{lem:vec=shift}, we know that the cyclic interval $S^{\pi}_{c_{1}}=(a_{1}-1,c_1]=[a_{1},c_{1}]$ contains the cyclic interval $S^{\sigma}_{c_{1}}=(b_{1}+1,c_{1}]=[b_{1}+2,c_{1}]$. Therefore, the CW-arrow $[b_{1}+1,c_{1}]$ of $\sigma^{:}$ has cardinality $|[a_{1},c_{1}]|-|[a_{1},b_{1}]|=k-|[a_{1},b_{1}]|$. Similarly, we may define $c_{2}=\pi_{k,n}(a_{2}-1)$ and deduce that the CW-arrow $[b_{2}+1,c_{2}]$ of $\sigma^{:}$ has cardinality $k-|[a_{2},b_{2}]|$.  

\textbf{Case 1:} the CW-arrows $[b_{1}+1,c_{1}]$ and $[b_{2}+1,c_{2}]$ are disjoint. Then by \Cref{thm:uniformCW}, the sum of their cardinalities, $(k-|[a_{1},b_{1}]|)+(k-|[a_{2},b_{2}]|)$, is at least $k+1$, and thus $|[a_{1},b_{1}]|+|[a_{2},b_{2}]|\leq k-1$.

\textbf{Case 2:} the CW-arrows $[b_{1}+1,c_{1}]$ and $[b_{2}+1,c_{2}]$ are not disjoint. Without loss of generality, assume that $b_{2}+1\in [b_{1}+1,c_{1}]$. Then, since $b_{1}+1\not\in [a_{2},b_{2}]$, we deduce that $[a_{2},b_{2}]\subseteq [b_{1}+2,c_{1}]$, and thus $|[a_{1},b_{1}]|+|[a_{2},b_{2}]|\leq |[a_{1},b_{1}]\cup[b_{1}+2,c_{1}]|=k-1$.

\textbf{The ``if'' direction:} assume $\sigma^{:}=\overrightarrow{\rho_{A}}(\pi_{k,n})$ for some $A\subseteq [n]$, and the union of any two distinct cyclic components of $A$ has cardinality at most $k-1$. By \cite[Theorem 26]{Ben19}, we know that $\rk(\sigma^{:})=k-1$. We then show that the union of any two CW-arrows of $\sigma^{:}$ has cardinality at least $k+1$, which implies $\sigma^{:}\lessdot \pi_{k,n}$ due to \Cref{thm:uniformCW}. Let $[b_{1}+1,c_{1}]$ and $[b_{2}+1,c_{2}]$ be the two CW arrows and assume, on the contrary, that the cardinality of their union is at most $k$. Then we must have $b_{1}+1\notin A$, since otherwise $c_{1}=\sigma(b_{1}+1)=(b_{1}+1+k\pmod n)$ and $|[b_{1}+1,c_{1}]|=k+1$. Similarly, we have $b_{2}+1\notin A$. We again divide our proof into two cases.

\textbf{Case 1:} the CW-arrows $[b_{1}+1,c_{1}]$ and $[b_{2}+1,c_{2}]$ are not disjoint. Without loss of generality, assume $b_{2}+1\in [b_{1}+1,c_{1}]$. Since $b_{1}+1,b_{2}+1\notin A$ and $\sigma^{:}=\overrightarrow{\rho_{A}}(\pi_{k,n})$, there exists some $x\in [b_{1}+1,b_{2}]$ such that $\sigma(b_{2}+1)=\pi_{k,n}(x)$, that is, $c_{2}=\pi_{k,n}(x)$. Thus, $|[b_{1}+1,c_{1}]\cup[b_{2}+1,c_{2}]|\geq \left|[b_{1}+1,x]\cup[x+1,\pi_{k,n}(x)]\right|\geq k+1$.

\textbf{Case 2:} the CW-arrows $[b_{1}+1,c_{1}]$ and $[b_{2}+1,c_{2}]$ are disjoint. Observe that $(\pi^{-1}_{k,n}(c_{1}),b_{1}]$ and $(\pi_{k,n}^{-1}(c_{2}),b_{2}]$ are both cyclic components of $A$. Therefore, 
\[|[b_{1}+1,c_{1}]|+|[b_{2}+1,c_{2}]|= k-\left|(\pi_{k,n}^{-1}(c_{1}),b_{1}]\right|+k-\left|(\pi_{k,n}^{-1}(c_{2}),b_{2}]\right|\geq 2k-(k-1)=k+1.\qedhere \]
\end{proof}

\begin{eg}
 Consider $\pi_{4,8} = 56781234$. If $A=\left\{1,3,5,8\right\}$, then the size of the union of any two components in $\{3\}, \{5\},[8,1]$ does not exceed $3 = 4 - 1$. By \Cref{thm:uniform}, we have $\overrightarrow{\rho_{1358}}\left(\pi_{4,8}\right)\lessdot \pi_{4,8}$. If $A = \left\{1,2,5,8\right\}$, then the union of two components $\left\{5\right\}$ and $[8,2]$ has size $4$, violating the condition. As a result, $\overrightarrow{\rho_{1258}}(\pi_{4,8})$ is not a quotient of $\pi_{4,8}$.
\end{eg}

\begin{remark}
The paper \cite{Ben19} also focuses on characterizing elementary positroid quotients of uniform matroids. They proved that the condition $|A|\leq k-1$ is sufficient for  $\overrightarrow{\rho_{A}}(\pi_{k,n})\lessdot \pi_{k,n}$ (\cite[Theorem 28]{Ben19}), while a necessary condition is that each \emph{individual} cyclic component of $A$ has cardinality at most $k-1$ (\cite[Theorem 26]{Ben19}). In comparison, our results show that the precise necessary and sufficient condition lies between these two: specifically, the union of any two cyclic components must have cardinality at most $k-1$.
\end{remark}
\section{Flag LPMs and necklace containment}\label{sec:5}

The main objective of this section is to examine the ``non-locality'' phenomenon in \Cref{thm:uniformCW}, as highlighted in \Cref{rmk:nonlocality}. To illustrate this concept, we introduce and prove \Cref{thm:LPM}, which serves as a contrasting example.

\subsection{Locality of necklace containment}

Recall that in the proof of \Cref{thm:shift}, we crucially utilized the observation that a flag positroid satisfies the ``necklace containment relation''. That is, for any flag positroid $(M,N)$, if the Grassmann necklace of $M$ and $N$ are $I^{\sigma}$ and $I^{\pi}$, respectively, we have $I^{\sigma}_{i}\subseteq I^{\pi}_{i}$ for all $i\in[n]$ (due to \Cref{prop:quotient-implies-containment} and \Cref{prop:necklace-minimizes}). By \Cref{def:conecklace}, we know that the Grassmann conecklaces $J^{\sigma}$ and $J^{\pi}$ of $M$ and $N$ also satisfy $J^{\sigma}_{i}\subseteq J^{\pi}_{i}$ for all $i\in [n]$. We record this observation as follows.

\begin{prop}\label{prop:Ne CoNe Cont}
    Given a flag positroid $(M,N)$, let $I^{\sigma},I^{\pi}$ be the Grassmann necklaces of $M$ and $N$, and let $J^{\sigma},J^{\pi}$ be the Grassmann conecklaces of $M$ and $N$, respectively. Then $I^{\sigma}_{i}\subseteq I^{\pi}_{i}$ and $J^{\sigma}_{i}\subseteq J^{\pi}_{i}$ for all $i\in [n]$.
\end{prop}

The perspective developed in \Cref{sec:3} also reveals that necklace containment relations can be translated into containment relations of Grassmann intervals (which are essentially equivalent to CW-arrows). We formalize this useful observation in the following proposition. The first item is merely a restatement of \Cref{cor:neckcontain=vecdom}, while the second follows easily from duality (see also the beginning of the proof of \Cref{A}).

\begin{prop}\label{prop:neck-coneck-containment}
Given decorated permutations $\sigma^{:}$ and $\pi^{:}$ over $[n]$, let $S^{\sigma}_{1},\dots,S^{\sigma}_{n}$ and $S^{\pi}_{1},\dots,S^{\pi}_{n}$ be their Grassmann intervals, as defined in \Cref{def:grassmann-matrix}. Then
\begin{enumerate}[label=(\arabic*)]
\item $I^{\sigma}_{i}\subseteq I^{\pi}_{i}$ for all $i\in [n]$ if and only if $S^{\sigma}_{i}\subseteq S^{\pi}_{i}$ for all $i\in [n]$.

\item $J^{\sigma}_{i}\subseteq J^{\pi}_{i}$ for all $i\in [n]$ if and only if $S^{\sigma}_{\sigma(i)}\subseteq S^{\pi}_{\pi(i)}$ for all $i\in [n]$.
\end{enumerate}
\end{prop}

We claim that both necklace containment $I^{\sigma}_{i}\subseteq I^{\pi}_{i}$ and $J^{\sigma}_{i}\subseteq J^{\pi}_{i}$ can be viewed as ``local'' criteria. As in \Cref{rmk:nonlocality}, consider the case where all CW-arrows of $\sigma^{:}$ and $\pi^{:}$ are of length much smaller than $n$. In other words, each $S^{\sigma}_{i}$ or $S^{\pi}_{i}$ covers only a few consecutive positions on the circle of elements in $[n]$. However, \Cref{thm:uniformCW} reveals that even when all CW-arrows are short, characterizing quotient relations requires accounting for multiple CW-arrows that may be arbitrarily far apart. Consequently, it is natural that ``local'' criteria, such as necklace and conecklace containment relations, fail to characterize quotient relations among positroids. We illustrate this with a concrete example.

\begin{eg}\label{cex:Oh conj}
Let $M$ be a positroid given by the decorated permutation $261534$ and take $N$ to be the uniform positroid $U_{4,6}$. The necklace and conecklace of $M$ are $(134,234,346,461,\allowbreak 561,613)$ and $(356,561,562,623,234,235)$ respectively. For $N$, they are $(1234,2345,3456,4561,\allowbreak 5612,6123)$ and $(3456,4561,5612,6123,1234,2345)$. They satisfy the containment conditions. However, $\rk_{M}([6])-\rk_{M}(1245)=1$ while $\rk_{N}([6])-\rk_{N}(1245)=0$. This means $M$ is \emph{not} a quotient of $N$, since the rank condition of quotient in \Cref{prop:rank-quotient} is violated. Therefore, the converse of \Cref{prop:Ne CoNe Cont} is not true. 
\end{eg}

\subsection{A conjecture of Oh and Xiang}

It was conjectured by Oh and Xiang \cite{Oh17} that flag positroids can be characterized by certain ``covering'' relations of CCW-arrows (\Cref{def:CCW-arrows}). 

\begin{conj}[{\cite[Conjecture 6.3]{Oh17}}]\label{conj:Oh-Xiang}
Let $M$ and $N$ be positroids over $[n]$ without loops or coloops. If every CCW-arrow of $M$ is the union of some CCW-arrows of $N$, then $M$ is a quotient of $N$. 
\end{conj}

It is not hard to see that \Cref{cex:Oh conj} disproves this conjecture. In fact, if $M$ is the positroid given by the decorated permutation $261534$, all CCW-arrows of $M$ have length at least 3. Since every cyclic interval of length 3 is a CCW-arrow of $U_{4,6}$, it follows that every CCW-arrow of $M$ is the union of some CCW-arrows of $U_{4,6}$. However, $M$ is not a quotient of $U_{4,6}$, as argued in \Cref{cex:Oh conj}.

\subsection{The case of LPMs}

Lattice path matroids (LPMs) are a special class of positroids that defined solely by the first term of its Grassmann necklace and conecklace (see \Cref{def:LPM}). By \Cref{prop:Ne CoNe Cont}, we know that if the LPM $M[U',L']$ is a quotient of $M[U,L]$, then $U'\subseteq U$ and $L'\subseteq L$. The work of Benedetti and Knauer \cite{Ben22} provides a necessary and sufficient criterion on $U',L',U$ and $L$ for which $M[U',L']$ is a quotient of $M[U,L]$. 

\begin{theorem}[{\cite[Theorem 19]{Ben22}}]
\label{thm:LPM-greedy-pairing}
Consider an LPM $M[U,L]$ over $[n]$. Order the elements of $U$ and $L$ as $u_{1}<\dots<u_{k}$ and $\ell_{1}<\dots<\ell_{k}$, respectively. Then an LPM $M[U',L']$ is a quotient of $M[U,L]$ if and only if $U'\subseteq U$, $L'\subseteq L$, and the elements of $U\setminus U'$ and $L\setminus L'$ can be ordered as $u_{i_{1}}<\dots<u_{i_{z}}$ and $\ell_{j_{1}}<\dots<\ell_{j_{z}}$, respectively, such that for each $s\in [z]$ one has $j_{s}\leq i_{s}$ and $u_{i_{s}}-\ell_{j_{s}}\leq i_{s}-j_{s}$.
\end{theorem}

The paper \cite{Ben22} also shows using this criterion that \Cref{conj:Oh-Xiang} holds for LPMs. That is, unlike the case of general positroids, if $M$ and $N$ are LPMs over $[n]$, the condition that every CCW-arrow of $M$ is the union of some CCW-arrows of $N$ suffices to ensure the quotient relation. We next show that for LPMs, in fact, even the weaker condition of necklace and conecklace containment suffices to ensure the quotient relation, that is, the converse of \Cref{prop:Ne CoNe Cont} is true in the case of LPMs.

\begin{theorem}[Formal statement of \Cref{thm:LPM}]\label{thm:LPM-formal}
Given a lattice path matroids $N'$ and $N$ over $[n]$. let $I',I$ be the Grassmann necklaces of $N'$ and $N$, and let $J',J$ be the Grassmann conecklaces of $N'$ and $N$, respectively. Then $N'$ is a quotient of $N$ if and only if $I'_{i}\subseteq I_{i}$ and $J'_{i}\subseteq J_{i}$ for all $i\in [n]$.
\end{theorem}

To prove \Cref{thm:LPM-formal}, we first apply \Cref{prop:neck-coneck-containment}. After that, the $\text{(iii)}\Rightarrow\text{(i)}$ direction of \cite[Proof of Theorem 41]{Ben22} essentially fills in the rest of the proof. For the sake of completeness, we provide a sketch of the proof (largely similar to \cite[Proof of Theorem 41]{Ben22}) here.

\begin{proof}[Proof of \Cref{thm:LPM-formal}]
The ``only if'' direction follows from \Cref{prop:Ne CoNe Cont}, so we focus on the ``if'' direction. Let $N'=M[U',L']$ and $N=M[U,L]$. By assumption we have $U'=I'_{1}\subseteq I_{1}=U$ and $L'=J'_{1}\subseteq J_{1}=L$. Order the elements of $U\setminus U'$ and $L\setminus L'$ as $u_{i_{1}}<\dots<u_{i_{z}}$ and $\ell_{j_{1}}<\dots<\ell_{j_{z}}$, respectively. Assume on the contrary that $M[U',L']$ is not a quotient of $M[U,L]$. Then by \Cref{thm:LPM-greedy-pairing}, for some $s\in [z]$ we have either $j_{s}>i_{s}$ or $u_{i_{s}}-\ell_{j_{s}}>i_{s}-j_{s}$. We consider the smallest $s$ such that this is the case.

\textbf{Case 1:} $j_{s}>i_{s}$. Let $\sigma^{:}$ and $\pi^:$ be the decorated permutation of $N'$ and $N$, respectively. It is easy to see that $\pi(\ell_{i_{s}})=u_{i_{s}}$. Let $x=\ell_{i_{s}}$. By the minimality of the choice of $s$, we must have $x\in L'$. Since $j_{s}>i_{s}$, the number of elements in $L'$ that are at most $x$ is at least $i_{s}-s+1$, and thus the number of elements in $U'$ that are at most $\sigma(x)$ is at least $i_{s}-s+1$. Therefore $\sigma(x)>u_{i_{s}}$, and thus we have $\pi(x)=u_{i_s}<\sigma(x)\leq x$. But by \Cref{prop:neck-coneck-containment}, we have $S^{\sigma}_{\sigma(x)}\subseteq S^{\pi}_{\pi(x)}$, leading to a contradiction (note that since $x\in L'$, if $\sigma(x)=x$ then $\sigma^{:}(x)=\underline{x}$ and $S^{\sigma}_{\sigma(x)}=[n]$).

\textbf{Case 2:} $u_{i_{s}}-\ell_{j_{s}}>i_{s}-j_{s}$. Let $\sigma^{:}$ and $\pi^:$ be the decorated permutation of $N'$ and $N$, respectively. Let $y$ be the smallest element of $[n]\setminus L$ greater than $\ell_{j_{s}}$, or equivalently, the $(\ell_{j_{s}}-j_{s}+1)$-th smallest element of $[n]\setminus L$. Such a $y$ exists because 
$$\ell_{j_{s}}-j_{s}\leq u_{i_{s}}-i_{s}-1\leq  n-|U|-1=n-|L|-1.$$
Now $\pi(y)$ is the $(\ell_{j_{s}}-j_{s}+1)$-th smallest element of $[n]\setminus U$. Since $u_{i_{s}}-i_{s}>\ell_{j_{s}}-j_{s}$, it follows that $u_{i_{s}}>\pi(y)$. This means the number of elements in $[n]\setminus U'$ that are at most $\pi(y)$ is at most $(\ell_{j_{s}}-j_{s}+1)+(s-1)$, while the number of elements in $[n]\setminus L'$ that are at most $y$ is at least $(\ell_{j_{s}}-j_{s}+1)+s$. Therefore, we have $\sigma(y)>\pi(y)\geq y$. But by \Cref{prop:neck-coneck-containment}, we have $S^{\sigma}_{\sigma(y)}\subseteq S^{\pi}_{\pi(y)}$, leading to a contradiction (note that since $y\not\in L$, if $\pi(y)=y$ then $\pi^{:}(y)=\overline{y}$ and $S^{\pi}_{\pi(y)}=\emptyset$).
\end{proof}

\Cref{thm:LPM-formal} (and \cite[Proof of Theorem 41]{Ben22}) stands in contrast to \Cref{cex:Oh conj}, demonstrating that in some sense, it is much easier to ensure quotient relations between LPMs than between general positroids. 

\begin{remark}
Note that in the proof of \Cref{thm:LPM-formal}, while the conecklace containment $J'_{i}\subseteq J_{i}$ is used for all $i\in [n]$, we used the necklace containment $I'_{i}\subseteq I_{i}$ only at $i=1$. By symmetry, we could have also used the necklace containment fully and the conecklace containment only at $i=1$. It is certainly impossible to not use the conecklace containment at all. For example, consider $M[14, 57]$ and $M[145, 467]$. In this case, we have 
\[
    I' =(14,24,34,45,56,16,17) \quad \text{and} \quad I=(145,245,345,456,156,167,147).
\]
So we have $I'_i\subseteq I_i$ for all $i=1,2, \dots, 7$, but $M[14, 57]$ is not a quotient of $M[145,467]$ because $J'_{1}\not\subseteq J_{1}$.
\end{remark}

\section*{Acknowledgement}

This research was carried out as part of the PACE program in the summer of 2023 at Peking University, Beijing, partially supported by NSFC grant 12426507. We thank Shiliang Gao, Yibo Gao, and Lauren Williams for helpful discussions.

\bibliographystyle{plain}
\bibliography{ref}
\end{document}